\documentclass[12pt]{amsart}

\usepackage{amsmath,amssymb,amsfonts,amsthm,latexsym,graphicx,multirow,enumerate,hyperref}
\usepackage{tikz}
\usetikzlibrary{calc}
\newcommand\A{\mathrm{A}} \newcommand\Aut{\mathrm{Aut}}
\newcommand\B{\mathrm{B}} \newcommand\BS{\mathrm{B}}
  \newcommand\Cay{\mathrm{Cay}}   
\newcommand\D{\mathrm{D}}

       \newcommand\Sy{\mathrm{S}} \newcommand\Sym{\mathrm{Sym}}

\newcommand\ZZ{\mathbb{Z}}

\newtheorem{theorem}{Theorem}[section]
\newtheorem{lemma}[theorem]{Lemma}
\newtheorem{proposition}[theorem]{Proposition}

\newtheorem{problem}[theorem]{Problem}
\newtheorem{conjecture}[theorem]{Conjecture}
\newtheorem{question}[theorem]{Question}
\theoremstyle{definition}

\newtheorem{example}[theorem]{Example}
\newtheorem{remark}[theorem]{Remark}

\usepackage{color}

\definecolor{Blue}{rgb}{0,0,1}
\definecolor{Red}{rgb}{1,0,0}
\definecolor{DarkGreen}{rgb}{0,0.6,0}
\definecolor{DarkYellow}{rgb}{1,1,0.2}
\definecolor{DarkPurple}{rgb}{.6,0,1}

\usepackage{xcolor}
\usepackage[normalem]{ulem}

\newcommand{\pmat}[1]{\begin{pmatrix}#1\end{pmatrix}}

\begin{document}

\title{Stability of circulant graphs}

\author[Qin]{Yan-Li Qin}
\address{Department of Mathematics\\Beijing Jiaotong University\\Beijing, 100044\\ P. R. China}
\email{yanliqin@bjtu.edu.cn}

\author[Xia]{Binzhou Xia}
\address{School of Mathematics and Statistics\\The University of Melbourne\\Parkville, VIC 3010\\Australia}
\email{binzhoux@unimelb.edu.au}

\author[Zhou]{Sanming Zhou}
\address{School of Mathematics and Statistics\\The University of Melbourne\\Parkville, VIC 3010\\Australia}
\email{sanming@unimelb.edu.au}

\maketitle

\begin{abstract}
The canonical double cover $\mathrm{D}(\Gamma)$ of a graph $\Gamma$ is the direct product of $\Gamma$ and $K_2$. If $\mathrm{Aut}(\mathrm{D}(\Gamma))=\mathrm{Aut}(\Gamma)\times\mathbb{Z}_2$ then $\Gamma$ is called stable; otherwise $\Gamma$ is called unstable. An unstable graph is nontrivially unstable if it is connected, non-bipartite and distinct vertices have different neighborhoods. In this paper we prove that every circulant graph of odd prime order is stable and there is no arc-transitive nontrivially unstable circulant graph. The latter answers a question of Wilson in 2008. We also give infinitely many counterexamples to a conjecture of Maru\v{s}i\v{c}, Scapellato and Zagaglia Salvi in 1989 by constructing a family of stable circulant graphs with compatible adjacency matrices.

\end{abstract}

\section{Introduction}

We study the stability of circulant graphs. Among others we answer a question of Wilson \cite{Wilson2008} and give infinitely many counterexamples to a conjecture of Maru\v{s}i\v{c}, Scapellato and Zagaglia Salvi \cite{MSZ1989}.

All graphs considered in the paper are finite, simple and undirected. As usual, for a graph $\Gamma$ we use $V(\Gamma)$, $E(\Gamma)$ and $\Aut(\Gamma)$ to denote its vertex set, edge set and automorphism group, respectively. For an integer $n \geqslant 1$, we use $n \Gamma$ to denote the graph consisting of $n$ vertex-disjoint copies of $\Gamma$. The complete graph on $n \geqslant 1$ vertices is denoted by $K_n$, and the cycle of length $n \geqslant 3$ is denoted by $C_n$. In this paper, we assume that each symbol representing a group or a graph actually represents the isomorphism class of the same. Thus, the statement ``$X=Y$" actually means that $X$ is isomorphic to $Y$, whether $X$ and $Y$ are both groups or both graphs. Further, if $X$ and $Y$ are groups, we write ``$X\leqslant Y$'' to indicate that $X$ is isomorphic to some subgroup of $Y$.

The \emph{canonical double cover} of a graph $\Gamma$ (see, for example, \cite{{LMS2015}}), denoted by $\D(\Gamma)$, is defined to be the direct product of $\Gamma$ and $K_2$. That is, $\D(\Gamma)$ is the graph with vertex set $V(\Gamma)\times\ZZ_2$ in which $(u,x)$ and $(v,y)$ are adjacent if and only if $u$ and $v$ are adjacent in $\Gamma$ and $x\ne y$. It can be verified that $\D(\Gamma)$ is connected if and only if $\Gamma$ is connected and non-bipartite (see, for example, \cite[Theorem~3.4]{BHM1980}). Clearly,
\begin{equation}\label{eq3}
\Aut(\D(\Gamma))\geqslant \Aut(\Gamma)\times\Aut(K_2)=\Aut(\Gamma)\times\ZZ_2.
\end{equation}
If $\Aut(\D(\Gamma))=\Aut(\Gamma)\times\ZZ_2$, then $\Gamma$ is called \emph{stable}; otherwise, $\Gamma$ is called \emph{unstable}. For example, $K_3$ is stable with $\D(K_3)$ isomorphic to $C_6$, while $C_4$ is unstable with $\D(C_4)$ isomorphic to $2C_4$.

While in general it is challenging to determine whether a graph is stable, one can easily see that the following graphs are all unstable: disconnected graphs; bipartite graphs with a nontrivial automorphism; graphs having two distinct vertices with the same neighborhood \cite[Proposition~4.1]{MSZ1989}. In light of these observations, we call a graph \emph{nontrivially unstable} if it is connected, non-bipartite, vertex-determining and unstable, where a graph is called \emph{vertex-determining} if distinct vertices have different neighborhoods. We are mainly concerned with nontrivially unstable graphs in the study of unstable graphs.

The stability of graphs was first studied in~\cite{MSZ1989} by Maru\v{s}i\v{c}, Scapellato and Zagaglia Salvi using the language of symmetric $(0,1)$ matrices. Since then this concept has been studied extensively by several authors from different viewpoints~\cite{LMS2015,MSZ1992,NS1996,Surowski2001,Surowski2003,Wilson2008}. In \cite{NS1996}, for example, the stability of graphs played an important role in finding regular embeddings of canonical double covers on orientable surfaces. In \cite{LMS2015}, close connections between the stability and two-fold automorphisms of graphs were found, and a method for constructing unstable graphs with large diameter such that every edge lies on a triangle was given. In \cite{MSZ1992}, searching for nontrivially unstable graphs led to the introduction of generalized Cayley graphs, and it was proved among others that every generalized Cayley graph which is not a Cayley graph is unstable. This result was then used in the same paper to construct an infinite family of nontrivially unstable graphs. In \cite{Surowski2001}, methods for constructing arc-transitive unstable graphs were given, and three infinite families of such graphs were constructed as applications.

Given a group $G$ and a nonempty inverse-closed subset $S$ of $G\setminus\{1\}$ (where $1$ is the identity element of $G$), the \emph{Cayley graph} of $G$ with \emph{connection set} $S$, denoted by $\Cay(G,S)$, is the graph with vertex set $G$ such that two vertices $x$ and $y$ are adjacent if and only if $yx^{-1}\in S$. A \emph{circulant graph}, or simply a \emph{circulant}, is a Cayley graph of a cyclic group. In~\cite[Theorems~C.1--C.4]{Wilson2008}, Wilson claimed that a circulant $\Cay(\ZZ_n,S)$ is unstable if one of the following conditions (C.1)--(C.4) holds (the condition that $n$ is even in (C.3) is implicit in the original statement in~\cite{Wilson2008}):
\begin{enumerate}[{\rm(C.1)}]
\item $n$ is even and has an even divisor $a$ such that for each even $s$ in $S$, $s+a$ is also in $S$;
\item $n$ is divisible by $4$ and has an odd divisor $b$ such that for each odd $s$ in $S$, $s+2b$ is also in $S$;
\item $n$ is even and $\ZZ_n$ has a subgroup $H$ such that $R:=\{j\bmod n\mid j\in S,\ j+H\nsubseteq S\} \neq \emptyset$, $D:=\gcd(R) > 1$, and $j/D$ is odd for each $j\in R$;
\item $n$ is even and there is an integer $g$ coprime to $n$ such that for each $s$ in $S$, $gs+n/2$ is also in $S$.
\end{enumerate}
In the same paper Wilson conjectured that for circulants the converse is also true:

\begin{conjecture}\label{conj2}
\emph{(\cite{Wilson2008})} For any nontrivially unstable circulant $\Cay(\ZZ_n,S)$, the parameters $n$ and $S$ satisfy at least one of \emph{(C.1)--(C.4)}.
\end{conjecture}

In Section~\ref{sec2}, we will show that, unfortunately, the claim in~\cite[Theorem~C.2]{Wilson2008} that all circulants satisfying (C.2) are unstable is not true. Nevertheless, along similar lines in the proof of~\cite[Theorem~C.2]{Wilson2008} one can show that $\Cay(\ZZ_n,S)$ is unstable if the following condition is satisfied:
\begin{enumerate}[{\rm(C.2$'$)}]
\item $n$ is divisible by $4$ and has an odd divisor $b$ such that for each odd $s$ in $S$, $s+2b$ is also in $S$, and for each $s\equiv0$ or $-b\pmod{4}$ in $S$, $s+b$ is also in $S$.
\end{enumerate}
In other words, one can restore \cite[Theorem~C.2]{Wilson2008} by replacing condition (C.2) in that theorem by (C.2$'$). However, it is worth noting that replacing (C.2) by (C.2$'$) in Conjecture~\ref{conj2} would make the conjecture false as there exist nontrivially unstable circulants satisfying none of (C.1), (C.2$'$), (C.3) and (C.4), as exemplified by the circulant $\Cay(\ZZ_{24}, S)$ where $S = \{2, 3, 8, 9, 10, 14,15, 16, 21, 22\}$. At this stage we do not have any conjectural classification of nontrivially unstable circulants. We would like to put forward the following problem.

\begin{problem}
\label{prob}
Classify all nontrivially unstable circulants.
\end{problem}

Note that if $n$ is odd then none of (C.1)--(C.4) can happen. Motivated by this we pose to tackle first the following conjecture which is weaker than Conjecture~\ref{conj2} and is a step towards resolving Problem \ref{prob}.

\begin{conjecture}\label{conj3}
There is no nontrivially unstable circulant of odd order.
\end{conjecture}

We prove that this weaker conjecture is true for circulants of prime order, as stated in the following result.

\begin{theorem}\label{thm2}
Every circulant of odd prime order is stable.
\end{theorem}

Using \textsc{Magma}~\cite{magma} we have searched all Cayley graphs of abelian groups of order up to $45$. Among them we have not found any nontrivially unstable graph of odd order. This provides further support to Conjecture \ref{conj3} and suggests to investigate whether there exists any nontrivially unstable Cayley graph of an abelian group of odd order.

In the study of the stability of graphs, arc-transitive graphs have received special attention due to their links to algebraic map theory~\cite{Surowski2001,Surowski2003}. A graph is called \emph{arc-transitive} or \emph{edge-transitive}, respectively, if the induced action of its automorphism group on its arc set or edge set is transitive. In general, arc-transitivity implies edge-transitivity, but the converse is not true. However, for Cayley graphs of abelian groups, in particular, for circulants, the two properties are equivalent (see Lemma~\ref{lem5}). In \cite{Wilson2008}, Wilson noticed that, although a large number of nontrivially unstable circulants were known, none of them was known to be arc-transitive (or, equivalently, edge-transitive). So he asked:

\begin{question}\label{ques1}
\emph{(\cite{Wilson2008})} Is there any arc-transitive nontrivially unstable circulant?
\end{question}

We answer this question by proving the following result.

\begin{theorem}
\label{thm1}
There is no arc-transitive nontrivially unstable circulant. In other words, a connected arc-transitive circulant is stable if and only if it is non-bipartite and vertex-determining.
\end{theorem}

Clearly, if a graph $\Gamma$ has adjacency matrix $A$ then its canonical double cover $\D(\Gamma)$ has adjacency matrix $\pmat{0&A\\A&0}$. In fact, this property can be used as an equivalent definition of the canonical double cover of a graph. The matrix $A$ is said to be \emph{compatible} with a nonidentity permutation matrix $P$ of the same size as $A$ if $AP$ is the adjacency matrix of some graph. We call $A$ \emph{compatible} if $A$ is compatible with at least one nonidentity permutation matrix. It can be verified that if $A$ is compatible then the adjacency matrix of $\Gamma$ under any ordering of $V(\Gamma)$ is also compatible. In other words, being compatible is independent of the underlying vertex-order in the adjacency matrix and hence is a graph-theoretic property. In \cite[Theorem~4.3]{MSZ1989}, Maru\v{s}i\v{c}, Scapellato and Zagaglia Salvi proved that every nontrivially unstable graph has a compatible adjacency matrix. Moreover, they conjectured that the converse is also true for connected non-bipartite vertex-determining graphs:

\begin{conjecture}
\label{conj1}
\emph{(\cite{MSZ1989})} A connected non-bipartite vertex-determining graph is unstable if and only if it has a compatible adjacency matrix.
\end{conjecture}

We disprove this conjecture by constructing an infinite family of counterexamples:

\begin{theorem}
\label{thm3}
Let $\ell>1$ and $m>1$ be coprime odd integers. Let $t$ be an integer such that $t\equiv-1\pmod{\ell}$ and $t\equiv1\pmod{m}$, and set $S=\{1,-1,t,-t\}$. Then $\Cay(\ZZ_{\ell m}, S)$ is connected, arc-transitive and stable with a compatible adjacency matrix.
\end{theorem}

The existence of $t$ above is ensured by the Chinese remainder theorem. Our proof of Theorem \ref{thm3} relies on Theorem \ref{thm1} and a characterization of compatibility of Cayley graphs (Lemma \ref{lem1}).

The rest of the paper is organized as follows. In the next section we will present some definitions, examples and preliminary results that will be used in later sections. In Section \ref{sec2}, we will discuss unstable circulants and prove Theorem \ref{thm2}. The proof of Theorems \ref{thm1} and \ref{thm3} will be given in Sections \ref{sec3} and \ref{sec4}, respectively.

\section{Preliminaries}

For a graph $\Gamma$ and a vertex $u$ of $\Gamma$, we use $N_\Gamma(u)$ to denote the neighborhood of $u$ in $\Gamma$. If $u$ and $v$ are adjacent in the graph under consideration, then we write $u \sim v$ and use $\{u,v\}$ to denote the edge between $u$ and $v$. Denote by $\overline{K}_n$ the graph of $n\geqslant1$ isolated vertices, and by $K_{n,n}$ the complete bipartite graph with $n\geqslant1$ vertices in each part of its bipartition.  As usual, the expression $X\subsetneq Y$ means that $X$ is a proper subset of a set $Y$.

\subsection{Product graphs}\label{sec1}

Let $\Sigma$ and $\Gamma$ be graphs. The \emph{direct product} $\Sigma\times\Gamma$, \emph{lexicographic product} $\Sigma[\Gamma]$, and \emph{Cartesian product} $\Sigma\Box\Gamma$ are graphs each with vertex set $V(\Sigma)\times V(\Gamma)$ and adjacency relation as follows: for $(u,x),(v,y)\in V(\Sigma)\times V(\Gamma)$, $(u,x)\sim(v,y)$ in $\Sigma\times\Gamma$ if and only if $u\sim v$ in $\Sigma$ and $x\sim y$ in $\Gamma$; $(u,x)\sim(v,y)$ in $\Sigma[\Gamma]$ if and only if $u\sim v$ in $\Sigma$, or $u=v$ and $x\sim y$ in $\Gamma$; $(u,x)\sim(v,y)$ in $\Sigma\Box\Gamma$ if and only if $u\sim v$ in $\Sigma$ and $x=y$, or $u=v$ and $x\sim y$ in $\Gamma$.

\begin{example}\label{examp3}
For any graph $\Sigma$ and integer $d>1$, the direct product $\Sigma\times K_d$ is isomorphic to $\Sigma[\overline{K}_d]-d\Sigma$, the lexicographic product of $\Sigma$ by $\overline{K}_d$ minus $d$ vertex-disjoint copies of $\Sigma$.
\end{example}

\begin{example}\label{examp2}
In particular, for $n\geqslant 3$, $K_n\times K_2 \cong K_2 \times K_n \cong K_{2}[\overline{K}_n]-n K_2 \cong K_{n,n}-nK_2$ is isomorphic to the complete bipartite graph $K_{n,n}$ minus a perfect matching $nK_2$. Thus $\Aut(K_n\times K_2)=\Sy_n\times\ZZ_2=\Aut(K_n)\times\ZZ_2$. Hence $K_n$ is stable.
\end{example}

\begin{lemma}\label{lem4}
Let $\Sigma$ and $\Gamma$ be graphs. Then $\Sigma\times\Gamma$ is vertex-determining if and only if $\Sigma$ and $\Gamma$ are both vertex-determining.
\end{lemma}

\begin{proof}
This follows immediately from the observation that for any $u\in V(\Sigma)$ and $x\in V(\Gamma)$ the neighborhood of $(u,x)$ in $\Sigma\times\Gamma$ is $N_\Sigma(u)\times N_\Gamma(x)$.
\end{proof}

\begin{lemma}\label{lem6}
Let $\Sigma$ be a graph with at least one edge, and let $d>1$ be an integer. Then $\Sigma[\overline{K}_d]$ is not vertex-determining.
\end{lemma}

\begin{proof}
Let $u$ be a vertex of $\Sigma$. Then for $x\in V(\overline{K}_d)$ the neighborhood of $(u,x)$ in $\Sigma[\overline{K}_d]$ is $N_\Sigma(u)\times V(\overline{K}_d)$. Thus, for distinct vertices $x$ and $y$ of $\overline{K}_d$, $(u,x)$ and $(u,y)$ have the same neighborhood in $\Sigma[\overline{K}_d]$. Therefore, $\Sigma[\overline{K}_d]$ is not vertex-determining.
\end{proof}

\subsection{Cartesian skeleton}\label{sec5}

The \emph{Boolean square} $\BS(\Gamma)$ of a graph $\Gamma$ is the graph with vertex set $V(\Gamma)$ and edge set $\{\{u,v\}\mid u, v \in V(\Gamma), u \ne v, N_{\Gamma}(u)\cap N_{\Gamma}(v)\neq\emptyset\}$. An edge $\{u,v\}$ of $\BS(\Gamma)$ is said to be \emph{dispensable with respect to $\Gamma$} if there exists $w\in V(\Gamma)$ such that
\[
N_{\Gamma}(u)\cap N_{\Gamma}(v)\subsetneq N_{\Gamma}(u)\cap N_{\Gamma}(w)\text{ or }N_{\Gamma}(u)\subsetneq N_{\Gamma}(w)\subsetneq N_{\Gamma}(v)
\]
and
\[
N_{\Gamma}(v)\cap N_{\Gamma}(u)\subsetneq N_{\Gamma}(v)\cap N_{\Gamma}(w)\text{ or }N_{\Gamma}(v)\subsetneq N_{\Gamma}(w)\subsetneq N_{\Gamma}(u).
\]
The \emph{Cartesian skeleton} $\Sy(\Gamma)$ of $\Gamma$ is the spanning subgraph of $\BS(\Gamma)$ obtained by removing from $\BS(\Gamma)$ all dispensable edges with respect to $\Gamma$.

The definitions above can be found in~\cite[Section~8.3]{HIK2011}, where a loop is required at each vertex of $\BS(\Gamma)$. We do not include loops in the definition of $\BS(\Gamma)$ above, but this will not affect subsequent discussions. Our definition of $\Sy(\Gamma)$ agrees with that in \cite[Section~8.3]{HIK2011}.

\begin{example}\label{skeleton}
Let $\Gamma=\Cay(\ZZ_8,\{1,4,7\})$. Then $\BS(\Gamma)=\Cay(\ZZ_8,\{2,3,5,6\})$ and $\Sy(\Gamma)=\Cay(\ZZ_8,\{3,5\})$, as illustrated in the following figures, where the dashed edges of $\B(\Gamma)$ are the dispensable edges with respect to $\Gamma$. In fact, for each dashed edge $\{i,i+2\}\in\B(\Gamma)$ with $i\in\ZZ_8$ we have
\[
N_\Gamma(i)\cap N_\Gamma(i+2)=\{i+1\}\subsetneq\{i+1,i+4\}=N_\Gamma(i)\cap N_\Gamma(i+5)
\]
and
\[
N_\Gamma(i+2)\cap N_\Gamma(i)=\{i+1\}\subsetneq\{i+1,i+6\}=N_\Gamma(i+2)\cap N_\Gamma(i+5),
\]
which implies that $\{i,i+2\}$ is a dispensable edge with respect to $\Gamma$.
\bigskip
\begin{center}
\begin{tikzpicture}
{\tiny
  \foreach \i in {0,1,2,3,4,5,6,7}
  {\coordinate (p\i) at (112.5-45*\i :1.2cm);
  \coordinate (q\i) at ($(112.5-45*\i :1.2cm)+(5cm,0)$) ;
  \coordinate (r\i) at ($(112.5-45*\i :1.2cm)+(10cm,0)$) ;
    };	
  \foreach \i/\j in {0/1,1/2,2/3,3/4,4/5,5/6,6/7,7/0,1/5,2/6,3/7,4/0}
  {
  \draw (p\i)--(p\j);
  };
  \foreach \i/\j in {0/3,1/4,2/5,3/6,4/7,5/0,6/1,7/2}
  {
  \draw (q\i)--(q\j);
  };
  \foreach \i/\j in {0/3,1/4,2/5,3/6,4/7,5/0,6/1,7/2}
  {
  \draw (r\i)--(r\j);
  };
  \foreach \i/\j in {0/2,1/3,2/4,3/5,4/6,5/7,6/0,7/1}
  {
  \draw[dashed] (q\i)--(q\j);
  };
  \foreach \i in {0,1,2,3,4,5,6,7}
  {
  \filldraw[fill=white] (p\i) circle (2pt);
  \node at (p\i) [label=112.5-45*\i :$\i$]{};
  \filldraw[fill=white] (q\i) circle (2pt);
  \node at (q\i) [label=112.5-45*\i :$\i$]{};
  \filldraw[fill=white] (r\i) circle (2pt);
  \node at (r\i) [label=112.5-45*\i :$\i$]{};
    };
}
  \node at (0,-2.0cm) {$\Gamma$};
  \node at (5cm,-2.0cm) {$\B(\Gamma)$};
  \node at (10cm,-2.0cm) {$\Sy(\Gamma)$};
\end{tikzpicture}
\end{center}
\end{example}

The next two lemmas are \cite[Proposition 8.10]{HIK2011} and \cite[Proposition 8.13(i)]{HIK2011}, respectively. (An $R$-thin graph as defined in \cite[Section~8.2]{HIK2011} is exactly a vertex-determining graph.)

\begin{lemma}\label{lem13}
If $\Sigma$ and $\Gamma$ are vertex-determining graphs without isolated vertices, then $\Sy(\Sigma\times\Gamma)=\Sy(\Sigma)\Box\Sy(\Gamma)$.
\end{lemma}

\begin{lemma}\label{lem14}
If $\Gamma$ is a connected non-bipartite graph, then $\Sy(\Gamma)$ is connected.
\end{lemma}

\subsection{Cayley graphs}

Let $G$ be a group and $S$ an inverse-closed subset of $G\setminus\{1\}$. Let $R\colon G\rightarrow\Sym(G)$ be the right regular representation of $G$ and let
\[
\Aut(G,S)=\{\alpha\in\Aut(G)\mid S^\alpha=S\}.
\]
It is well known and straightforward to verify that
\begin{equation}\label{eq1}
R(G)\rtimes\Aut(G,S)\leqslant\Aut(\Cay(G,S)).
\end{equation}
If the equality in~\eqref{eq1} holds, then $\Cay(G,S)$ is called \emph{normal}; otherwise, it is called \emph{nonnormal}.

The following result is well known in the literature. We give its proof for the completeness of this paper.

\begin{lemma}\label{lem3}
If $\Aut(G,S)$ is transitive on $S$, then $\Cay(G,S)$ is arc-transitive.
\end{lemma}

\begin{proof}
Note that $\Aut(G,S)$ is a subgroup of $\Aut(\Cay(G,S))$ stabilizing the vertex $1$. If $\Aut(G,S)$ is transitive on $S$, then the stabilizer in $\Aut(\Cay(G,S))$ of the vertex $1$ is transitive on its neighborhood. In this case, $\Cay(G,S)$ is arc-transitive since $R(G)\leqslant\Aut(\Cay(G,S))$ is transitive on the vertex set of $\Cay(G,S)$.
\end{proof}

\begin{lemma}\label{lem5}
A Cayley graph of an abelian group is arc-transitive if and only if it is edge-transitive.
\end{lemma}

\begin{proof}
Let $G$ be an abelian group and $\Cay(G,S)$ a Cayley graph of $G$. We only prove that edge-transitivity implies arc-transitivity for $\Cay(G,S)$ since the converse is obvious. Since $G$ is abelian and $S$ is inverse-closed, the inverse map $\tau\colon x \mapsto x^{-1},\ x \in G$, is in $\Aut(G,S)$ and hence in $\Aut(\Cay(G,S))$. Thus, for adjacent vertices $x$ and $y$, $R(x^{-1})\tau R(y)$ is in $\Aut(\Cay(G,S))$ and it swaps $x$ and $y$. Therefore, if $\Cay(G,S)$ is edge-transitive, then it must be arc-transitive.
\end{proof}

\begin{lemma}\label{lem11}
Let $G=H\times K$ be a group, where $H$ is a subgroup and $K$ is a characteristic subgroup of order at least $5$. Suppose that $S=T\times(K\setminus\{1\})$ is inverse-closed, where $T\subseteq H$. Then $\Cay(G,S)$ is nonnormal.
\end{lemma}

\begin{proof}
Suppose for a contradiction that $\Gamma=\Cay(G,S)$ is normal. View $V(\Gamma)$ as the Cartesian product of $H$ and $K$. Then the action of $R(K)$ on the second coordinate gives rise to a regular subgroup $L$ of $\Sym(K)$. Since $S=T\times(K\setminus\{1\})$, the action of $\Sym(K)$ on the second coordinate induces a subgroup $M$ of $\Aut(\Gamma)$. Since $R(G)$ is normal in $\Aut(\Gamma)$ and $R(K)$ is characteristic in $R(G)$, we infer that $R(K)$ is normal in $\Aut(\Gamma)$ and thus normalized by $M$. It follows that $L$ is normalized by $\Sym(K)$. However, $\Sym(K)$ does not have any regular normal subgroup, a contradiction.
\end{proof}

The following characterization of connected arc-transitive nonnormal circulants was obtained independently by Kov\'{a}cs~\cite{Kovacs2004} and Li~\cite{Li2005}. Here we rephrase the statement of~\cite[Theorem~1]{Kovacs2004} in light of Example~\ref{examp3}.

\begin{proposition}\label{prop1}
Let $\Gamma$ be a connected arc-transitive nonnormal circulant of order $n$. Then one of the following holds:
\begin{enumerate}[{\rm(a)}]
\item $\Gamma=K_n$;
\item $\Gamma=\Sigma[\overline{K}_d]$, where $n=md$, $d>1$ and $\Sigma$ is a connected arc-transitive circulant of order $m$;
\item $\Gamma=\Sigma\times K_d$, where $n=md$, $d>3$, $\gcd(m,d)=1$ and $\Sigma$ is a connected arc-transitive circulant of order $m$.
\end{enumerate}
\end{proposition}

A Cayley graph $\Cay(G,S)$ is called a \emph{CI-graph} if for every Cayley graph $\Cay(G,T)$ with $\Cay(G,S)\cong\Cay(G,T)$ there exists $\sigma\in\Aut(G)$ such that $T=S^\sigma$. The next lemma can be derived from Proposition~\ref{prop1} (see~\cite[Section~7.3]{Li2002}).

\begin{lemma}\label{lem9}
Every arc-transitive circulant is a CI-graph.
\end{lemma}

\section{Unstable circulants}
\label{sec2}

For $n$ and $S$ satisfying condition (C.i) in the Introduction, where i $\in\{1,2,3,4\}$, Wilson stated in~\cite[Theorem~C.i]{Wilson2008} that the circulant $\Cay(\ZZ_n,S)$ is unstable. However, computer search using \textsc{Magma} gives a number of counterexamples to~\cite[Theorem~C.2]{Wilson2008}. That is, there are stable circulants $\Cay(\ZZ_n,S)$ satisfying~(C.2). For example, the circulant $\Cay(\ZZ_{12},\{3,4,8,9\})$ satisfies~(C.2) with $b=3$ but is stable. In fact, for $n=12$ and $b=3$, there are $31$ connection sets $S$ satisfying condition~(C.2) but only $22$ of them give rise to unstable circulants.

We now prove Theorem~\ref{thm2}.

\begin{proof}
Let $p$ be an odd prime and $\Gamma=\Cay(\ZZ_p,S)$ a Cayley graph of $\ZZ_p$, where $S\neq\emptyset$. Clearly, $\Gamma$ is connected and non-bipartite, and so $\D(\Gamma)$ is connected. If $\Gamma=K_p$, then from Example~\ref{examp2} we already know that $\Gamma$ is stable. Assume that $\Gamma$ is not a complete graph in the following. As $p$ is odd, $|S|$ is an even number between $1$ and $p-2$. So $\D(\Gamma)$ is a connected graph of even valency between $1$ and $p-2$. Note that $\D(\Gamma)=\Gamma\times K_2=\Cay(\ZZ_p\times\ZZ_2, S\times\{1\})$. So the isomorphism $\ZZ_p\times\ZZ_2\cong\ZZ_{2p}$ implies that $\D(\Gamma)$ is a Cayley graph of $\ZZ_{2p}$. Hence,  by~\cite[Theorem~1.6]{DWX1998}, either $\D(\Gamma)$ is a normal Cayley graph of $\ZZ_{2p}$ or $\D(\Gamma)=\Sigma[\overline{K}_2]$ for some graph $\Sigma$.

First assume that $\D(\Gamma)$ is a normal Cayley graph of $\ZZ_{2p}=\ZZ_p\times\ZZ_2$. Then
\begin{equation}\label{eq8}
\Aut(\D(\Gamma))=R(\ZZ_p\times\ZZ_2)\rtimes\Aut(\ZZ_p\times\ZZ_2,S\times\{1\}).
\end{equation}
Since $p$ is odd, we have $\Aut(\ZZ_p\times\ZZ_2)=\Aut(\ZZ_p)\times\Aut(\ZZ_2)$ and so
\[
\Aut(\ZZ_p\times\ZZ_2,S\times\{1\})=\Aut(\ZZ_p,S)\times\Aut(\ZZ_2)=\Aut(\ZZ_p,S)\times1.
\]
This together with~\eqref{eq8} yields
\[
|\Aut(\D(\Gamma))|=2p|\Aut(\ZZ_p,S)|=2|R(\ZZ_p)\rtimes\Aut(\ZZ_p,S)|\leqslant|\Aut(\Gamma)\times\ZZ_2|.
\]
In view of \eqref{eq3} we conclude that $\Aut(\D(\Gamma))=\Aut(\Gamma)\times\ZZ_2$. Hence $\Gamma$ is stable.

Next assume that $\D(\Gamma)=\Sigma[\overline{K}_2]$ for some graph $\Sigma$. Then $E(\Sigma)\neq\emptyset$ and so Lemma~\ref{lem6} implies that $\D(\Gamma)$ is not vertex-determining. Since $\D(\Gamma)=\Gamma\times K_2$ and $K_2$ is vertex-determining, it follows from Lemma~\ref{lem4} that $\Gamma$ is not vertex-determining. So there exist distinct elements $a$ and $b$ of $\ZZ_p$ such that $a$ and $b$ have the same neighborhood in $\Cay(\ZZ_p,S)$. This means that $S+a=S+b$, or equivalently, $S+(a-b)=S$. Note that $\langle a-b\rangle=\ZZ_p$. We conclude that $S=\ZZ_p$, a contradiction. The proof is thus completed.
\end{proof}

\section{Proof of Theorem \ref{thm1}}
\label{sec3}

The following lemma is from~\cite[Proposition~4.2]{MSZ1989} (see also~\cite[Theorem~3.2]{LMS2015}).

\begin{lemma}\label{lem12}
A graph $\Gamma$ is unstable if and only if there exist distinct permutations $\alpha$ and $\beta$ of $V(\Gamma)$ such that for all $u,v\in V(\Gamma)$,
\begin{equation}\label{eq7}
u\sim v\Leftrightarrow u^\alpha\sim v^\beta.
\end{equation}
\end{lemma}

Recall the Boolean square $\BS(\Gamma)$ and the Cartesian skeleton $\Sy(\Gamma)$ of a graph $\Gamma$ defined in Section~\ref{sec5}.

\begin{lemma}\label{lem15}
Let $\Gamma$ be a graph, and let $\alpha$ and $\beta$ be (not necessarily distinct) permutations of $V(\Gamma)$ satisfying~\eqref{eq7} for all $u,v\in V(\Gamma)$. Then the following statements hold:
\begin{enumerate}[{\rm(a)}]
\item $N_{\Gamma}(w^\alpha)=(N_{\Gamma}(w))^\beta$ and $N_{\Gamma}(w^\beta)=(N_{\Gamma}(w))^\alpha$ for each $w\in V(\Gamma)$;
\item $\alpha,\beta\in\Aut(\BS(\Gamma))$;
\item $\alpha,\beta\in\Aut(\Sy(\Gamma))$.
\end{enumerate}
\end{lemma}

\begin{proof}
Let $x\in N_{\Gamma}(w^\alpha)$. Then $\{w^\alpha,x\}\in E(\Gamma)$, and as $\beta$ is a permutation of $V(\Gamma)$, we have $x=u^{\beta}$ for some $u\in V(\Gamma)$. It follows from~\eqref{eq7} that $\{w,u\}\in E(\Gamma)$, which means that $u\in N_\Gamma(w)$. Hence $x=u^\beta\in(N_\Gamma(w))^\beta$. This shows that $N_{\Gamma}(w^\alpha)\subseteq(N_\Gamma(w))^\beta$. Let $y\in(N_\Gamma(w))^\beta$. Then $y=v^\beta$ for some $v\in N_\Gamma(w)$. Since $\{w,v\}\in E(\Gamma)$, we deduce from~\eqref{eq7} that $\{w^\alpha,v^\beta\}\in E(\Gamma)$ and so $y=v^\beta\in N_{\Gamma}(w^\alpha)$. This shows that $(N_{\Gamma}(w))^\beta\subseteq N_{\Gamma}(w^\alpha)$. Thus $N_{\Gamma}(w^\alpha)=(N_{\Gamma}(w))^\beta$. Similarly, we have $N_{\Gamma}(w^\beta)=(N_{\Gamma}(w))^\alpha$, completing the proof of statement~(a).

For $u,v\in V(\Gamma)$, by statement~(a) we have
\begin{equation}\label{eq10}
(N_\Gamma(u)\cap N_\Gamma(v))^\beta=(N_\Gamma(u))^\beta\cap(N_\Gamma(v))^\beta=N_\Gamma(u^\alpha)\cap N_\Gamma(v^\alpha),
\end{equation}
and so
\begin{align*}
\{u,v\}\in E(\B(\Gamma))&\Leftrightarrow N_\Gamma(u)\cap N_\Gamma(v)\neq\emptyset\\
&\Leftrightarrow(N_\Gamma(u)\cap N_\Gamma(v))^\beta\neq\emptyset \\
&\Leftrightarrow N_\Gamma(u^\alpha)\cap N_\Gamma(v^\alpha)\neq\emptyset\\
&\Leftrightarrow\{u^\alpha,v^\alpha\}\in E(\B(\Gamma)).
\end{align*}
Hence $\alpha\in\Aut(\BS(\Gamma))$. Similarly, we have $\beta\in\Aut(\BS(\Gamma))$, which completes the proof of statement~(b).

Since $\alpha\in\Aut(\BS(\Gamma))$ and $\Gamma$ is finite, to prove $\alpha\in\Aut(\Sy(\Gamma))$ it suffices to show that $\{u^\alpha,v^\alpha\}\in E(\BS(\Gamma))\setminus E(\Sy(\Gamma))$ for each $\{u,v\}\in E(\BS(\Gamma))\setminus E(\Sy(\Gamma))$. Let $\{u,v\}$ be an edge of $\BS(\Gamma)$ that is dispensable with respect to $\Gamma$. Then there exists $w\in V(\Gamma)$ such that
\begin{equation}\label{eq5}
N_{\Gamma}(u)\cap N_{\Gamma}(v)\subsetneq N_{\Gamma}(u)\cap N_{\Gamma}(w)\text{ or }N_{\Gamma}(u)\subsetneq N_{\Gamma}(w)\subsetneq N_{\Gamma}(v)
\end{equation}
and
\begin{equation}\label{eq6}
N_{\Gamma}(v)\cap N_{\Gamma}(u)\subsetneq N_{\Gamma}(v)\cap N_{\Gamma}(w)\text{ or }N_{\Gamma}(v)\subsetneq N_{\Gamma}(w)\subsetneq N_{\Gamma}(u).
\end{equation}
It follows from~\eqref{eq5} that
\[
(N_{\Gamma}(u)\cap N_{\Gamma}(v))^\beta\subsetneq(N_{\Gamma}(u)\cap N_{\Gamma}(w))^\beta\text{ or }
(N_{\Gamma}(u))^\beta\subsetneq(N_{\Gamma}(w))^\beta\subsetneq(N_{\Gamma}(v))^\beta.
\]
Then by~\eqref{eq10} and statement~(a),
\[
N_{\Gamma}(u^\alpha)\cap N_{\Gamma}(v^\alpha)\subsetneq N_{\Gamma}(u^\alpha)\cap N_{\Gamma}(w^\alpha)\text{ or }
N_{\Gamma}(u^\alpha)\subsetneq N_{\Gamma}(w^\alpha)\subsetneq N_{\Gamma}(v^\alpha).
\]
In the similar vein, we derive from~\eqref{eq6} that
\[
N_{\Gamma}(v^\alpha)\cap N_{\Gamma}(u^\alpha)\subsetneq N_{\Gamma}(v^\alpha)\cap N_{\Gamma}(w^\alpha)\text{ or }N_{\Gamma}(v^\alpha)\subsetneq N_{\Gamma}(w^\alpha)\subsetneq N_{\Gamma}(u^\alpha).
\]
Therefore, $\{u^\alpha,v^\alpha\}$ is an edge of $\BS(\Gamma)$ that is dispensable with respect to $\Gamma$. This proves $\alpha\in\Aut(\Sy(\Gamma))$. Similarly, $\beta\in\Aut(\Sy(\Gamma))$, whence statement~(c) holds.
\end{proof}

The next lemma will play an important role in our proof of Theorem \ref{thm1}. It is also of interest in its own right.

\begin{lemma}\label{lem8}
Let $\Sigma$ be a graph of order $m$ and let $d>2$ be an integer coprime to $m$. If $\Sigma\times K_d$ is nontrivially unstable, then $\Sigma$ is nontrivially unstable.
\end{lemma}

\begin{proof}
Let $\Gamma=\Sigma\times K_d$. Suppose that $\Gamma$ is nontrivially unstable. Since $\Gamma$ is connected, $\Sigma$ is connected. Since $\Gamma$ is non-bipartite, $\Gamma$ contains an odd cycle and so $\Sigma$ contains an odd cycle, which implies that $\Sigma$ is non-bipartite. Moreover, since $\Gamma$ is vertex-determining, we derive from Lemma~\ref{lem4} that $\Sigma$ is vertex-determining. To complete the proof it remains to prove that $\Sigma$ is unstable.

As $\Gamma$ is unstable, by Lemma~\ref{lem12} there exist $\alpha,\beta\in\Sym(V(\Gamma))$ with $\alpha\neq\beta$ such that $u\sim v$ if and only if $u^\alpha\sim v^\beta$ for all $u,v\in V(\Gamma)$. So by Lemma~\ref{lem15}(c) we have $\alpha,\beta\in \Aut(\Sy(\Gamma))$. Since $\Sigma$ and $K_d$ are connected and vertex-determining, we derive from Lemma~\ref{lem13} that
\[
\Sy(\Gamma)=\Sy(\Sigma\times K_d)=\Sy(\Sigma)\Box\Sy(K_d)=\Sy(\Sigma)\Box K_d.
\]
Since $\Sigma$ is connected and non-bipartite, Lemma~\ref{lem14} ensures that $\Sy(\Sigma)$ is connected. Since
\[
\gcd(|V(\Sy(\Sigma))|,|V(K_d)|)=\gcd(m,d)=1,
\]
it then follows from~\cite[Corollary~6.12]{HIK2011} that
\[
\Aut(\Sy(\Gamma))=\Aut(\Sy(\Sigma)\Box K_d)=\Aut(\Sy(\Sigma))\times\Aut(K_d).
\]
As a consequence, we have $\alpha=(\alpha_1,\alpha_2)$ and $\beta=(\beta_1,\beta_2)$ for some $\alpha_1,\beta_1\in\Sym(V(\Sigma))$ and $\alpha_2,\beta_2\in\Sym(V(K_d))$. Now for $u_1,u_2\in V(\Sigma)$ and $v_1,v_2\in V(K_d)$,
\begin{equation}\label{eq9}
(u_1,v_1)\sim(u_2,v_2)\Leftrightarrow(u_1,v_1)^{\alpha}\sim(u_2,v_2)^\beta
\Leftrightarrow({u_1}^{\alpha_1},{v_1}^{\alpha_2})\sim({u_2}^{\beta_1},{v_2}^{\beta_2}).
\end{equation}
Fixing $v_1,v_2\in V(K_d)$ such that $v_1\sim v_2$, we deduce from~\eqref{eq9} that
\[
u_1\sim u_2\Rightarrow(u_1,v_1)\sim(u_2,v_2)\Rightarrow({u_1}^{\alpha_1},{v_1}^{\alpha_2})\sim({u_2}^{\beta_1},{v_2}^{\beta_2})
\Rightarrow{u_1}^{\alpha_1}\sim{u_2}^{\beta_1},
\]
while fixing $v_1,v_2\in V(K_d)$ such that ${v_1}^{\alpha_2}\sim{v_2}^{\beta_2}$, we deduce from~\eqref{eq9} that
\[
{u_1}^{\alpha_1}\sim{u_2}^{\beta_1}\Rightarrow({u_1}^{\alpha_1},{v_1}^{\alpha_2})\sim({u_2}^{\beta_1},{v_2}^{\beta_2})
\Rightarrow(u_1,v_1)\sim(u_2,v_2)\Rightarrow u_1\sim u_2.
\]
This shows that
\[
u_1\sim u_2 \Leftrightarrow {u_1}^{\alpha_1}\sim{u_2}^{\beta_1}.
\]
Similarly, we have
\[
v_1\sim v_2 \Leftrightarrow {v_1}^{\alpha_2}\sim{v_2}^{\beta_2}.
\]
If $\alpha_2\neq\beta_2$, then Lemma~\ref{lem12} would imply that $K_d$ is unstable, contradicting Example~\ref{examp2}. Hence $\alpha_2=\beta_2$. Since $\alpha\neq\beta$, we then obtain $\alpha_1\neq\beta_1$ and so by Lemma~\ref{lem12}, $\Sigma$ is unstable.
\end{proof}

We are now in a position to prove Theorem \ref{thm1}:

\begin{proof}
Suppose for a contradiction that $\Gamma=\Cay(\ZZ_n,S)$ is arc-transitive and nontrivially unstable with minimum order $n$. Then $\Gamma$ is connected, non-bipartite, vertex-determining and unstable. By Example~\ref{examp2}, $\Gamma$ is not a complete graph. Then we derive from Lemma~\ref{lem6} and Proposition~\ref{prop1} that either $\Gamma$ is a normal Cayley graph, or $\Gamma=\Gamma_1\times K_{d}$ with $n=md$, $d>3$, $\gcd(m,d)=1$ and $\Gamma_1$ a connected arc-transitive circulant of order $m$. In the latter case, Lemma~\ref{lem8} implies that $\Gamma_1$ is arc-transitive and nontrivially unstable, contradicting the minimality of $\Gamma$. Thus $\Gamma$ is a normal Cayley graph. Since $\Gamma$ is arc-transitive, we conclude that $\Aut(\ZZ_n,S)$ is transitive on $S$.

Suppose $n$ is even. Then each automorphism of $\ZZ_n$ is induced by the multiplication by an odd integer. Moreover, as $\Gamma$ is connected, there exists $s\in S$ with $s$ odd. Therefore, $S$ is the orbit of $\Aut(\ZZ_n,S)$ containing $s$ and so has only odd elements. This implies that $\Gamma$ is bipartite, a contradiction. Hence $n$ is odd.

Note that $\D(\Gamma)=\Gamma\times K_2=\Cay(\ZZ_n\times\ZZ_2,S\times\{1\})$. Thus the isomorphism $\ZZ_n\times\ZZ_2\cong\ZZ_{2n}$ implies that $\D(\Gamma)$ is a circulant of order $2n$. Moreover, since $\Gamma$ and $K_2$ are both arc-transitive, $\D(\Gamma)$ is also arc-transitive. Since $\D(\Gamma)$ is not a complete graph, we then derive from Proposition~\ref{prop1} that one of the following holds:
\begin{enumerate}[{\rm(i)}]
\item $\D(\Gamma)$ is a normal Cayley graph of $\ZZ_n\times\ZZ_2$;
\item $\D(\Gamma)=\Sigma[\overline{K}_c]$, where $2n=\ell c$, $c>1$ and $\Sigma$ is a connected arc-transitive circulant of order $\ell$;
\item $\D(\Gamma)=\Sigma\times K_c$, where $2n=\ell c$, $c>3$, $\gcd(\ell,c)=1$ and $\Sigma$ is a connected arc-transitive circulant of order $\ell$.
\end{enumerate}

First assume that~(i) occurs. Then
\begin{equation}\label{eq2}
\Aut(\D(\Gamma))=R(\ZZ_n\times\ZZ_2)\rtimes\Aut(\ZZ_n\times\ZZ_2,S\times\{1\}).
\end{equation}
Since $n$ is odd, we have $\Aut(\ZZ_n\times\ZZ_2)=\Aut(\ZZ_n)\times\Aut(\ZZ_2)$ and hence
\[
\Aut(\ZZ_n\times\ZZ_2,S\times\{1\})=\Aut(\ZZ_n,S)\times\Aut(\ZZ_2)=\Aut(\ZZ_n,S)\times1.
\]
This together with~\eqref{eq2} yields
\[
|\Aut(\D(\Gamma))|=2n|\Aut(\ZZ_n,S)|=2|R(\ZZ_n)\rtimes\Aut(\ZZ_n,S)|\leqslant|\Aut(\Gamma)\times\ZZ_2|.
\]
In view of~\eqref{eq3} we then conclude that $\Aut(\D(\Gamma))=\Aut(\Gamma)\times\ZZ_2$ and hence $\Gamma$ is stable, a contradiction.

Next assume that~(ii) occurs. Then by Lemma~\ref{lem6}, $\D(\Gamma)$ is not vertex-determining. However, $\D(\Gamma)=\Gamma\times K_2$ and $K_2$ is vertex-determining. Hence Lemma~\ref{lem4} implies that $\Gamma$ is not vertex-determining, a contradiction.

Finally assume that~(iii) occurs. Then $\Sigma\times K_c=\D(\Gamma)=\Gamma\times K_2$ is bipartite and thus does not contain any odd cycle. This implies that $\Sigma$ does not contain any odd cycle. Hence $\Sigma$ is bipartite. As a consequence, $\ell$ is even. Write $\Sigma=\Cay(\ZZ_\ell,S_1)$. Then $\D(\Gamma)$ is isomorphic to $\Cay(\ZZ_\ell\times\ZZ_c,S_1\times S_2)$ with $S_2=\ZZ_c\setminus\{0\}$. It is easy to check that the map $(x,y)\mapsto(1-n)x+ny$ is a well-defined isomorphism from $\ZZ_n\times\ZZ_2$ to $\ZZ_{2n}$. Since $2n=\ell c$, $\gcd(\ell,c)=1$ and $c=2n/\ell$ divides $n$, we then have an isomorphism
\[
\varphi\colon\ZZ_n\times\ZZ_2\rightarrow\ZZ_\ell\times\ZZ_c,\ (x,y)\mapsto(((1-n)x+ny)\bmod\ell,x\bmod c),
\]
which induces a graph isomorphism from $\D(\Gamma)=\Cay(\ZZ_n\times\ZZ_2,S\times\{1\})$ to $\Cay(\ZZ_\ell\times\ZZ_c,(S\times\{1\})^\varphi)$. Therefore,
\[
\Cay(\ZZ_\ell\times\ZZ_c,S_1\times S_2)\cong\Cay(\ZZ_\ell\times\ZZ_c,(S\times\{1\})^\varphi).
\]
Since $\Cay(\ZZ_\ell\times\ZZ_c,S_1\times S_2)\cong\D(\Gamma)$ is an arc-transitive circulant, we deduce from Lemma~\ref{lem9} that
\begin{equation}\label{eq4}
(S\times\{1\})^\varphi=S_1^{\sigma_1}\times S_2^{\sigma_2}=S_1^{\sigma_1}\times S_2
\end{equation}
for some $\sigma_1\in\Aut(\ZZ_\ell)$ and $\sigma_2\in\Aut(\ZZ_c)$. Let $T=\{t\bmod(\ell/2)\mid t\in S_1^{\sigma_1}\}\subseteq\ZZ_{\ell/2}$. Note that $n=\ell c/2$ and $\gcd(\ell/2,c)=1$. We have an isomorphism
\[
\psi\colon\ZZ_n\rightarrow\ZZ_{\ell/2}\times\ZZ_c,\ z\mapsto(z\bmod(\ell/2),z\bmod c).
\]
For each $s\in S$, we derive from~\eqref{eq4} that
\[
(((1-n)s+n)\bmod\ell,s\bmod c)=(s,1)^\varphi\in S_1^{\sigma_1}\times S_2,
\]
whence
\[
s^\psi=(((1-n)s+n)\bmod(\ell/2),s\bmod c)\in T\times S_2.
\]
This implies that $S^\psi\subseteq T\times S_2$. On the other hand, as $\psi$ is an isomorphism,
\[
|S^\psi| = |S|=|S\times\{1\}|=|(S\times\{1\})^\varphi|=|S_1^{\sigma_1}\times S_2|=|S_1^{\sigma_1}||S_2|\geqslant|T||S_2|.
\]
Therefore, $S^\psi=T\times S_2$. Since $\ell$ is even and $\gcd(\ell,c)=1$, we see that $c$ is odd and hence $c\geqslant5$. Now Lemma~\ref{lem11} shows that $\Gamma=\Cay(\ZZ_n,S)$ is nonnormal, a contradiction. This completes the proof.
\end{proof}

\begin{remark}
The statement in Theorem~\ref{thm1} is not true for general Cayley graphs of abelian groups. That is, there are arc-transitive nontrivially unstable Cayley graphs of abelian groups. For example, for $G=\langle a,b\mid a^4=b^4=1,ab=ba\rangle\cong\ZZ_4\times\ZZ_4$ and $S=\{a^2b^2,b^2,ab^{-1},a^{-1}b,b,b^{-1}\}$, the Cayley graph $\Cay(G,S)$ is arc-transitive but nontrivially unstable.
\end{remark}

\section{Proof of Theorem \ref{thm3}}
\label{sec4}

First we give the following characterization of compatibility of Cayley graphs.

\begin{lemma}\label{lem1}
The adjacency matrix of a Cayley graph $\Cay(G,S)$ is compatible if and only if there exists a nonidentity $\sigma\in\Sym(G)$ such that $x^\sigma x^{-1}\notin S$ and
\[
y^\sigma x^{-1}\in S\Leftrightarrow x^\sigma y^{-1}\in S
\]
for any $x,y\in G$.
\end{lemma}

\begin{proof}
Fix an order of the elements of $G$ and let $A$ be the adjacency matrix of $\Cay(G,S)$ under this order. The entry $A_{x,y}$ of $A$ in row $x$ and column $y$ is $1$ if and only if $yx^{-1}\in S$. Note that $A$ is compatible if and only if there exists a nonidentity permutation matrix $P$ such that $(AP)_{x,x}=0$ and
\[
(AP)_{x,y}=1\Leftrightarrow(AP)_{y,x}=1
\]
for any $x,y\in G$. Hence $A$ is compatible if and only if there exists a nonidentity permutation $\sigma$ on $G$ such that $A_{x,x^\sigma}=0$ and
\[
A_{x,y^\sigma}=1\Leftrightarrow A_{y,x^\sigma}=1
\]
for any $x,y\in G$. Consequently, $A$ is compatible if and only if there exists a nonidentity $\sigma\in\Sym(G)$ such that $x^\sigma x^{-1}\notin S$ and
\[
y^\sigma x^{-1}\in S\Leftrightarrow x^\sigma y^{-1}\in S
\]
for any $x,y\in G$, as stated in the lemma.
\end{proof}

We are now ready to prove Theorem \ref{thm3}:

\begin{proof}
Let $\ell$, $m$, $t$ and $S = \{1,-1,t,-t\}$ be as in Theorem \ref{thm3}. Denote $G = \ZZ_{\ell m}$ and $\Gamma = \Cay(G, S)$.

\smallskip
\textit{Claim 1}: $\Gamma$ is non-bipartite and vertex-determining.
\smallskip

In fact, as $|G|=\ell m$ is odd, $\Gamma$ is non-bipartite. Suppose that $\Gamma$ is not vertex-determining. Then there exist distinct elements $a$ and $b$ of $G$ such that $a$ and $b$ have the same neighborhood in $\Gamma$. Hence $S+a=S+b$. So $S+(a-b)=S$ and $S$ is a union of left cosets of $\langle a-b\rangle$ in $G$. Therefore, $|S|$ is divisible by $|\langle a-b\rangle|$. However, this is impossible as $|S|=4$ while $\langle a-b\rangle$ is a nontrivial subgroup of the group $G$ of odd order. Thus $\Gamma$ is vertex-determining. This proves Claim 1.

\smallskip
\textit{Claim 2}: Define $x^\sigma=tx$ for $x\in G$. Then $\sigma$ is an involution in $\Aut(G)$ such that $x^\sigma-x\notin S$ and
\[
y^\sigma-x\in S\Leftrightarrow x^\sigma-y\in S
\]
for any $x,y\in G$.
\smallskip

In fact, since $t\equiv-1\pmod{\ell}$ and $t\equiv1\pmod{m}$, we know that $t$ is coprime to $\ell m$ and is not congruent to $1$ modulo $\ell m$. Thus $\sigma$ is a nonidentity element of $\Aut(G)$. Moreover, $t^2\equiv1\pmod{\ell m}$ and hence $\sigma^2$ is the identity. Therefore, $\sigma$ is an involution in $\Aut(G)$.

Let $x$ and $y$ be arbitrary elements of $G$. Since $t-1$ is divisible by $m$ and no element of $S$ lies in $\langle m\rangle$, it follows that $x^\sigma-x=(t-1)x\in\langle m\rangle$ and so $x^\sigma-x\notin S$. Note that $S$ is inverse-closed and setwise stabilized by the involution $\sigma$. We have
\[
y^\sigma-x\in S\Leftrightarrow(y^\sigma-x)^\sigma\in S\Leftrightarrow y^{\sigma^2}-x^\sigma\in S
\Leftrightarrow y-x^\sigma\in S\Leftrightarrow x^\sigma-y\in S.
\]
This completes the proof of Claim 2.

\smallskip
\textit{Claim 3}: $\Gamma$ is connected and arc-transitive.
\smallskip

Clearly, $\langle S\rangle=G$. Hence $\Gamma$ is connected. By Claim 2, the map $\sigma\colon x\mapsto tx$ is a nonidentity automorphism of $G$ such that the condition in Lemma~\ref{lem1} holds. It follows from Lemma~\ref{lem1} that the adjacency matrix of $\Gamma$ is compatible. Moreover, by Claim 2, $\sigma$ is an involution, and so $\sigma$ fixes $S$ and is transitive on $\{1,t\}$ and $\{-1,-t\}$, respectively. Let $\tau$ be the automorphism of $G$ sending each element to its inverse. Then $\tau$ fixes $S$ and is transitive on $\{1,-1\}$ and $\{t,-t\}$, respectively. Thus $\Aut(G,S)\geqslant\langle\sigma,\tau\rangle$ is transitive on $S$, and so by Lemma~\ref{lem3}, $\Gamma$ is arc-transitive. This proves Claim 3.

By Claims 1 and 3, the circulant $\Gamma$ is connected, non-bipartite, vertex-determining and arc-transitive. Now appealing to Theorem~\ref{thm1} we obtain that $\Gamma$ is stable.
\end{proof}

\vskip0.1in
\noindent\textsc{Acknowledgements.} This work was done during the first author's visit to The University of Melbourne. The first author would like to thank Beijing Jiaotong University for financial support for this visit and National Natural Science Foundation of China (11671030) for financial support during her PhD program. The first author is very grateful to Professor Jin-Xin Zhou for suggesting the research topic and would like to thank Professor Jin-Xin Zhou and Professor Yan-Quan Feng for helpful advices. The authors would like to thank the anonymous referees for their very valuable comments.


\begin{thebibliography}{}

\bibitem{BHM1980}
R. A. Brualdi, F. Harary and Z. Miller, Bigraphs versus digraphs via matrices,
\textit{J. Graph Theory}  4  (1980), no. 1, 51--73.

\bibitem{magma}
W. Bosma, J. Cannon and C. Playoust, The Magma Algebra System I: The
User Language, \textit{J. Symbolic Comput.} 24 (1997), no. 3--4, 235--265.

\bibitem{DWX1998}
S.-F. Du, R.-J. Wang and M.-Y. Xu, On the normality of Cayley digraphs of groups of order twice a prime,
\textit{Australas. J. Combin.}  18  (1998), 227--234.

\bibitem{HIK2011}
R. Hammack, W. Imrich, S. Klav\v{z}ar,
\textit{Handbook of Product Graphs}, CRC Press 2011.

\bibitem{Kovacs2004}
I. Kov\'{a}cs, Classifying arc-transitive circulants,
\textit{J. Algebraic Combin.}  20  (2004),  no. 3, 353--358.

\bibitem{Li2002}
C. H. Li, On isomorphisms of finite Cayley graphs---a survey,
\textit{Discrete Math.}   256  (2002),  no. 1--2, 301--334.

\bibitem{Li2005}
C. H. Li, Permutation groups with a cyclic regular subgroup and arc transitive circulants,
\textit{J. Algebraic Combin.}  21  (2005),  no. 2, 131--136.

\bibitem{LMS2015}
J. Lauri, R. Mizzi and R. Scapellato, Unstable graphs: a fresh outlook via TF-automorphisms,
\textit{Ars Math. Contemp.}  8  (2015),  no. 1, 115--131.

\bibitem{MSZ1989}
D. Maru\v{s}i\v{c}, R. Scapellato and N. Zagaglia Salvi, A characterization of particular symmetric $(0,1)$ matrices,
\textit{Linear Algebra Appl.}  119 (1989), 153--162.

\bibitem{MSZ1992}
D. Maru\v{s}i\v{c}, R. Scapellato and N. Zagaglia Salvi, Generalized Cayley graphs,
\textit{Discrete Math.}  102  (1992),  no. 3, 279--285.

\bibitem{NS1996}
R. Nedela and M. \v{S}koviera, Regular embeddings of canonical double coverings of graphs,
\textit{J. Combin. Theory Ser. B}  67  (1996),  no. 2, 249--277.

\bibitem{Surowski2001}
D. Surowski, Stability of arc-transitive graphs,
\textit{J. Graph Theory}  38  (2001),  no. 2, 95--110.

\bibitem{Surowski2003}
D. Surowski, Automorphism groups of certain unstable graphs,
\textit{Math. Slovaca}  53  (2003),  no. 3, 215--232.

\bibitem{Wilson2008}
S. Wilson, Unexpected symmetries in unstable graphs,
\textit{J. Combin. Theory Ser. B}  98  (2008),  no. 2, 359--383.

\bibliographystyle{100}

\end{thebibliography}
\end{document}